%% file: paper_hypersurfaces.tex
\documentclass[a4paper,10pt, english]{amsart}

\usepackage{amsmath,amssymb,amsthm,mathtools}

\usepackage[centering]{geometry}

\input{packages}

\title[Hypersurfaces of any homogeneous $\CP^3$]{Hypersurfaces of any homogeneous $\CP^3$}

\author{Michaël Liefsoens}

\address{M. Liefsoens, KU\ Leuven, Department of Mathematics, Celestij\-nenlaan 200B -- Box 2400, 3001 Leuven, Belgium}
\email{michael.liefsoens@kuleuven.be}
\thanks{M. Liefsoens is supported by the Research Foundation Flanders (FWO) with project 11PG324N}

\subjclass[2020]{53C15, 53C55, 53C30, 53C42}
\keywords{nearly Kähler manifold, almost product structure, hypersurface}

\begin{document}

\begin{abstract}
	Hypersurfaces are studied and classified under multiple additional assumptions in any Riemannian homogeneous space $(\CP^3, g_a)$, including nearly Kähler $\CP^3$. Notably, all extrinsically homogeneous hypersurfaces are classified in all these spaces, with an explicit family of examples. Moreover, for nearly Kähler $\CP^3$, all Hopf hypersurfaces are classified. Finally, Codazzi-like hypersurfaces (and in particular parallel and totally geodesic hypersurfaces), totally umbilical hypersurfaces and constant sectional curvature hypersurfaces are proven to not exist in any homogeneous $\CP^3$. 
\end{abstract}

\maketitle

\section{Introduction}

In the four homogeneous nearly Kähler six-manifolds, hypersurfaces have received significant attention. Mostly in $\sphere{6}$ and the nearly Kähler $\sphere{3}\times \sphere{3}$, much work has been done. In \Cite{anarella2024}, an overview is given of (among others) the study of hypersurfaces in the four homogeneous nearly Kähler six-manifolds. It is noteworthy to mention that many studies have been undertaken in the nearly Kähler $\sphere{3}\times \sphere{3}$, with many (non-existence) results in the literature. Examples include non-existence of totally umbilical hypersurfaces \Cite{hu2018}; hypersurfaces with parallel second fundamental form \Cite{hu2018}; and the study of (Hopf) hypersurfaces with a number of distinct principal curvatures \Cite{hu2020a,yao2022}.

However, in the nearly Kähler flag manifold and nearly Kähler $\CP^3$, only \Citeauthor{deschamps2021} studied hypersurfaces in \Cite{deschamps2021}. They considered the almost contact structure $\phi$ induced from the almost complex structure $J$ given by 
\begin{equation}\label{eq:almost_contact_structure}
    \phi(X) = J X - g(JX,N)N,
\end{equation}
where $N$ is the unit normal.
They considered when this tensor $\phi$ commutes with the shape operator of the unit normal $A_N$. This condition implies being nearly cosymplectic ($\nabla \phi$ skew-symmetric) and being Hopf ($J N$ is an eigenvector of $A_N$, or, equivalently, the integral curve of $JN$ is totally geodesic). They proved that no hypersurface of the flag manifold, nor of nearly Kähler $\CP^3$, satisfies $A_N \circ \phi = \phi \circ A_N$. Moreover, they also excluded totally geodesic and totally umbilical hypersurfaces in both the flag manifold and the nearly Kähler $\CP^3$. 

In this work, the main focus is the nearly Kähler $\CP^3$ and the main result reads:
\begin{mainthm}\label{thm:main_theorem_NK_CP3}
    Let $\mathcal{H}$ be a hypersurface in the nearly Kähler $\CP^3$ with unit normal $N$, and let $\pi:\sphere{7}\to \CP^3$ be the Hopf fibration.
    The following are equivalent
    \begin{enumerate}
        \item $\mathcal{H}$ is Hopf,
        \item $\mathcal{H}$ is Hopf for both Kähler and nearly Kähler $\CP^3$,
        \item $\mathcal{H}$ is Hopf and horizontal, i.e. $N$ is horizontal with respect to the twistor fibration $\tau: \CP^3 \to \sphere{4}$,
        \item $\mathcal{H}$ is extrinsically homogeneous,
        \item $\mathcal{H}$ has constant principal curvature and is horizontal,
        \item $\mathcal{H}$ is the pre-image of a geodesic three sphere in $\sphere{4}$,
        \item $\mathcal{H}$ is congruent to one of the examples $\sphere{3}\times\sphere{3}\to \CP^3 : (p,q) \mapsto \pi(\cos  t\, p, \sin t\, q)$. 
    \end{enumerate}

    Let $\phi$ be the contact structure induced on $\mathcal{H}$, cfr. \Cref{eq:almost_contact_structure}. 
    Then, moreover, the following are equivalent
    \begin{enumerate}
        \item $\mathcal{H}$ is Hopf with $A_N JN = 0$
        \item $\mathcal{H}$ is Hopf and minimal
        \item $\phi A_N + A_N \phi = 0$
        \item $\mathcal{H}$ is the pre-image under the twistor fibration of a totally geodesic three sphere in $\sphere{4}$.
        \item $\mathcal{H}$ is congruent to the example $\sphere{3}\times\sphere{3}\to \CP^3 : (p,q) \mapsto \pi\l( \frac{1}{\sqrt{2}} (p, q) \r)$. 
    \end{enumerate}
\end{mainthm}

Using the description of \Cite{liefsoens2024}, many results can be extended to all homogeneous metrics on $\CP^3$ (denoted by $(\CP^3, g_a)$) as the nearly Kähler structure is interwoven tightly with the other metric structures. As such, all extrinsically homogeneous hypersurfaces of $(\CP^3, g_a)$ can be classified. Explicitly, consider the examples
\begin{equation}\label{eq:family_Hopf_hypersurfaces_intro}
    \pi \circ \psi_t: \sphere{3}\times\sphere{3} \to \CP^3: (p,q) \mapsto \pi( \cos t \, p, \sin t\, q ), \qquad t \in (0, \pi/2),
\end{equation} 
where $\pi:\sphere{7}\to \CP^3$ is the Hopf fibration, then we have 
\begin{mainthm}
    Let $g_a$ be any homogeneous metric on $\CP^3$, except for the Fubini-Study metric (which corresponds to $a=1$), then all extrinsically homogeneous hypersurfaces of $(\CP^3,g_a)$ are locally congruent to one of the examples of \Cref{eq:family_Hopf_hypersurfaces_intro}. In particular, they are all Hopf and horizontal. 
\end{mainthm}
We will observe that the family of \Cref{eq:family_Hopf_hypersurfaces_intro} are all extrinsically different, yet display an intrinsic mirror symmetry around the unique minimal family member. 

Finally, multiple non-existence results are obtained in all $(\CP^3, g_a)$. They can be summarised as follows.
\begin{mainthm}
    The homogeneous $(\CP^3, g_a)$ does not admit any Codazzi-like (and in particular no parallel or totally geodesic); totally umbilical; or constant sectional curvature hypersurfaces.
\end{mainthm}

\section{Preliminaries}

\subsection{Fundamental equations of hypersurfaces}\label{sec:hypersurfaces_general}

Suppose we have a hypersurface $\mathcal{H}$ of a Riemannian manifold $(M,g)$ with unit normal $N$. Let $R$ be the Riemann tensor of $(M,g)$, and $R^\mathcal{H}$ the curvature tensor of $\mathcal{H}$ with the pull-back metric coming from $g$. 

With respect to the hypersurface and the metric $g$, we have the second fundamental form $h$ as the part of the ambient Levi-Civita connection that is orthogonal to the hypersurface. Given the unit normal, we also have the shape-operator $A$, given by $g(h(X,Y), N) = g(A X, Y)$. Furthermore, we define $\alpha(X,Y) = g(h(X,Y),N)$. The Gauss and Codazzi equation read
\begin{subequations}
\begin{align}
    g(R^\mathcal{H}(X,Y)Z,W) &= g(R(X,Y)Z,W) + \alpha(X,W) \alpha(Y,Z) - \alpha(X,Z) \alpha(Y,W), \label{eq:Gauss}\\ 
    g( R(X,Y)Z , N) &= g( (\nabla^\perp h)(X,Y,Z) - (\nabla^\perp h)(Y, X,Z), N), \label{eq:Codazzi}
\end{align} 
\end{subequations} respectively, where $X,Y,Z,W$ are tangent vector fields of $\mathcal{H}$.

\subsection{Types of hypersurfaces}

We recall some standard definitions for different types of hypersurfaces of Riemannian manifolds $(M,g)$ and of Riemannian manifolds with an almost complex structure $(M, g, J)$. Suppose $N$ is the unit normal of the hypersurface and denote the identity component of the isometry group of $(M,g)$ by $\Iso(M,g)_\circ$.

The hypersurface is called Hopf when $J N$ is a principal direction, i.e. it is an eigenvector of the shape operator. Another way of thinking about this, is that the integral curve of $JN$ is totally geodesic. 

A submanifold is said to be extrinsically homogeneous if there is a closed subgroup $G$ of $\Iso(M,g)_\circ$ such that the submanifold is the orbit of a point under this subgroup. 

We say that a hypersurface is totally umbilical, when $g(h(X,Y),N)=g(\nabla_X Y,N)$ is proportional to $g(X,Y)$ for all tangent vectors $X,Y$. If the second fundamental form satisfies the condition that $\nabla h$ is a totally symmetric tensor, we say that the hypersurface is Codazzi or Codazzi-like. The Codazzi hypersurfaces include the parallel ($\nabla h = 0$) and totally geodesic ($h=0$) hypersurfaces. Minimal submanifolds are those with vanishing mean curvature, i.e. the trace of $h$ vanishes. Recall the standard notation $(X \wedge Y)Z = g(Y,Z)X - g(X,Z)Y$. Finally, a hypersurface is said to have constant sectional curvature if and only if $R^\mathcal{H}(X,Y)Z = c (X \wedge Y)Z$ for some constant $c$.

\subsection{(Nearly Kähler) \texorpdfstring{$\CP^3$}{ℂP³} via the Hopf fibration}\label{sec:description}

Recall the Hopf fibration $\pi: \sphere{7} \to \CP^3$, which is a Riemannian submersion when $\sphere{7}$ is equipped with its round metric and $\CP^3$ with its Fubini-Study metric $\FSm$. There is an almost complex structure $\Jo$ on $\CP^3$ inherited from multiplication by $i$ in $\complexs^4$. Moreover, recall that $(\CP^3, \FSm)$ is Einstein, has constant holomorphic sectional curvature $4$, and its isometry group is $\set{ \CP^3 \to \CP^3: \pi(x) \mapsto \pi(U x) \mid U \in \SU(4) }$.

Nearly Kähler $\CP^3$ can be described as the twistor space over $\sphere{4} \cong \quaternions P^1$, where $\quaternions$ are the quaternions. We denote the twistor fibration with $\tau$. In \Cite{liefsoens2024}, a more concrete description is given of nearly Kähler $\CP^3$, and of $\CP^3$ with any homogeneous metric. There, the seven-sphere is embedded in $\quaternions^2$, and the following distributions are considered: $V(p) = ip$, $\lift{\Dtwo}(p) = \Span\set{ j p, k p }$ and $\lift{\Dfour} = (V \oplus \lift{\Dtwo})^{\perp} \subset T\sphere{7}$. It is proven that $\Dtwo = \dd \pi \lift{\Dtwo}$ and $\Dfour = \dd \pi \lift{\Dfour}$ are well defined and orthogonal for the Fubini-Study metric on $\CP^3$. Moreover, $\Dtwo$ is integrable and its leafs are the fibers of the twistor fibration. 

An almost product structure $P$ was defined in \Cite{liefsoens2024} by requiring that it acts as the identity on $\Dfour$ and as minus the identity on $\Dtwo$, and by linearity.
From the almost product structure, a new almost complex structure is defined as $J = P \Jo = \Jo P$. By taking the metric 
\[ g(X,Y) = \frac{3}{2} \FSm(X,Y)+ \frac{1}{2} \FSm(X,P Y), \] the space $(\CP^3, g, J)$ is a nearly Kähler manifold, i.e. $G(X,Y) = (\nabla J)(X,Y)$ is skew-symmetric. The metrics $g_a$, given by
\[ g_a(X,Y) = \frac{1+a}{2} \FSm(X,Y) + \frac{a-1}{2} \FSm(P X,Y), \]
exhaust all the homogeneous metrics on $\CP^3$ (up to global rescaling), see \cite{onishchik1963,volper1999}. Note that for $a=2$, we have the unique metric such that $(\CP^3, g_a, J)$ is nearly Kähler, i.e. $g_2 = g$; and that for $a=1$, we have the Fubini-Study metric $\CP^3$. 

The Levi-Civita connections $\FSnabla$, $\nabla$, $\nabla^a$ of $\FSm$, $g$, and $g_a$ are linked as $\nabla_X Y = \FSnabla_X Y + D(X,Y)$ and $\nabla^a_X Y = \FSnabla_X Y + D^a(X,Y)$, where $D$ and $D^a$ are given by 
\begin{align*}
    D(X,Y) &= \frac{1}{2} \mathcal{P}_2 G( \Jo X, Y), &
    D^a(X,Y) &= \frac{a-1}{a} \mathcal{P}_2 G( \Jo X, Y),
\end{align*}
and $2 \mathcal{P}_2 = \identity +P$. Moreover, for all $a\neq 1$, the isometries of $(\CP^3, g_a)$ have been characterised in \Cite{liefsoens2024} as those Kähler isometries that commute with $P$. Finally, the following expression for the Riemann curvature was given for $(\CP^3, g_a)$:
\begingroup
\allowdisplaybreaks
\begin{align*}
    R^a(X,Y)Z 
    &= \frac{(a-1)(a+2)}{a^2} (X \wedge_a Y)Z
    \\
    &\quad
    + \frac{1}{a}\bigg( (X \wedge_a Y)Z + (\Jo X \wedge_a \Jo Y)Z + 2g_a(X,\Jo  Y) \Jo  Z \bigg) \\
    &\quad
    + \frac{1-a}{a^2}\bigg( (X \wedge_a Y)Z  + (J X \wedge_a J Y)Z + 2 g_a(X,JY)JZ \bigg) \\
    &\quad
    + \frac{1-a}{a}\bigg( (X \wedge_a Y)P Z + P (X \wedge_a Y)Z  - \frac{a+2}{a}P( X \wedge_a Y)P Z \bigg)  .
\end{align*}
\endgroup

Finally, we recall some useful equations that hold for the different structures on the nearly Kähler $\CP^3$. First, we have that it is of constant type, with the constant given by unity with the choices made in this work. We then have
\begin{subequations}
\begin{align}
    \norm{\nabla J(X,Y)}^2 &= \norm{X}^2\norm{Y}^2 - g(X,Y)^2 - g(X, JY)^2 \\ 
    \nabla G(X, Y,Z) &= J (Y\wedge Z)X  - g(JY,Z)X  \\
    G(X,G(Y,Z)) &= (Y\wedge Z)X  + J (Y\wedge Z)J  \\ 
    g(G(X,Y),G(Z,W)) &= g\l( (Z\wedge W)Y, X\r) + g\l( J(Z\wedge W)JY, X\r).
 \end{align} 
 \end{subequations}
Second, we have the covariant derivative of the Kähler almost complex structure
\begin{align}
    (\nabla \Jo)(X,Y) &= G(\mathcal{P}_1 X, \mathcal{P}_2 Y).
\end{align}
Finally, we remark that $(\nabla^a J)$ can be written as the sum of $G = (\nabla J)$ (which is skew-symmetric) and a symmetric tensor as follows
\[ (\nabla^a J)(X,Y) = G^a(X,Y) = G(X,Y) + \frac{2-a}{a} \mathcal{P}_2 G(\Jo X, Y). \]
Note that for $a=2$ (and only then), the symmetric part vanishes, so that indeed $(\CP^3, g_2, J)$ is nearly Kähler.

\subsection{Complex almost contact structure on \texorpdfstring{$\CP^3$}{ℂP³}}\label{sec:contact_structure}

Recall that $\CP^3$ is endowed with a complex almost contact structure. Let $U \in \Dtwo$, $V = - \Jo U$ and $X,Y$ be arbitrary vector fields. Let $\sigma(X) = \FSm( \FSnabla_X U, V)$; $u(X) = \FSm(X,U)$ and $v(X) = \FSm(X,V)$. In \Cite{blair2000}, they prove the existence of a (local) structure $\Phi$, and associated to it $\Psi = \Jo \Phi$, so that the following properties hold for all vector fields $X,Y$:
\begin{subequations} 
    \begin{gather}
        \Phi U = \Phi V = 0  \hspace{3cm}
        \Phi^2  = \Psi^2 = -\identity + u \otimes U + v \otimes V \label{eq:psi_only_on_D24_and_square_psi}  \\ 
        \FSm(X,\Phi Y) = - \FSm(\Phi X,Y) \hspace{3cm}
        \Phi \Jo = -\Jo \Phi  \label{eq:psi_antiCommutes_JO} \\
        \FSnabla_X U = -\Psi X + \sigma(X) V \hspace{3cm}
        \FSnabla_X V = -\Phi X - \sigma(X) U \label{eq:nabla0_UV} \\ 
        (\FSnabla_X \Psi) Y = \FSm(X,Y) U - u(Y)X - \FSm(X,J Y)V-v(Y)JX + \sigma(X)\Phi Y \label{eq:nabla0_psi} \\
        (\FSnabla_X \Phi) Y = \FSm(X,Y) V - v(Y)X - \FSm(X,J Y)U+u(Y)JX - \sigma(X)\Psi Y \label{eq:nabla0_phi}.
    \end{gather}
\end{subequations}
Moreover, in \Cite{blair2000} it is also discussed how these structures change when changing charts. The structures $(u,v,\Phi, \Psi)$ are referred to as the complex almost contact structure of $\CP^3$. 

An alternative way of seeing this complex almost contact structure of $\CP^3$, is to use the skew-symmetric tensor $G = \nabla J$ of nearly Kähler $\CP^3$. 
Let $A$ be an unit-length vector field in $\mathcal{D}^2$. For all $(\CP^3,g_a)$ with $a \neq 1$, define
\begin{align*}
    \Phi_A X &\coloneqq JG(A, X) &
    \Psi_A X &\coloneqq J \Phi_{A} X = \Phi_{-JA} X =  - G(A, X).
\end{align*}
For each such $A$, the structures $\Phi_A$ and $\Psi_A$ behave as complex almost contact structures.

\section{Angle function for hypersurfaces}\label{sec:angle}

Let $\mathcal{H}$ be a hypersurface of $(\CP^3, g_a)$ ($a\neq 1$) with unit normal $N_a$. Define the angle function $\theta_a$ by 
\[ \cos(2 \theta_a) = g_a(JN_a, \Jo N_a) = g_a(N_a, PN_a), \]
where at any one point, $\theta_a$ can be chosen to lie between $0$ and $\pi/2$.
Since for all $a\neq 1$, every isometry of $(\CP^3,g_a)$ preserves the almost product structure $P$, this angle function is an invariant of a hypersurface in $(\CP^3,g_a)$. This allows to introduce the following terminology.
\begin{definition}
    Let $a \neq 1$. A hypersurface of $(\CP^3,g_a)$ is called
    \begin{itemize}
        \item \textit{horizontal}, if $\theta_a$ is identically zero; 
        \item \textit{vertical}, if $\theta_a$ is identically $\pi/2$;
        \item $g_a$-\textit{isotropic}, if $\theta_a$ is identically $\pi/4$.
    \end{itemize}
\end{definition}
\begin{remark}
    Note that if a hypersurface is horizontal for $(\CP^3,g_a)$, it is automatically horizontal for $(\CP^3,g_b)$, for any $a, b \neq 1$. The same statement holds when horizontal is replaced by vertical. 
\end{remark}

We have the following immediate characterisations for these notions.
\begin{proposition}\label{thm:characterisation_special_values}
    Let $\mathcal{H}$ be a hypersurface in $(\CP^3,g_a)$ ($a\neq 1$) with unit normal $N_a$. Then,
    \begin{itemize}
        \item $\mathcal{H}$ is horizontal if and only if $N_a \in \Dfour$;
        \item $\mathcal{H}$ is vertical if and only if $N_a \in \Dtwo$;
        \item $\mathcal{H}$ is $g_a$-isotropic if and only if $N_a, JN_a, \Jo N_a, PN_a$ are two-by-two $g_a$-orthogonal.
    \end{itemize}
\end{proposition}
The terminology for horizontal and vertical hypersurfaces comes from this characterisation: a hypersurface is horizontal (vertical), if the normal is horizontal (vertical) for the twistor fibration $\tau$.

We immediately have the following general result.
\begin{proposition}\label{thm:no_vertical_hypersurfaces}
    No hypersurface of $(\CP^3,g_a)$ can be vertical.
\end{proposition}
\begin{proof}
    It is sufficient to prove this statement for the nearly Kähler $(\CP^3,g_2,J) = (\CP^3, g,J)$. 
    Take any unit $g$-length vector field $N \in \Dtwo$. Note that then also $JN \in \Dtwo$. Take any $X \in \Dfour$. Let $Y = JG(X,N)$ and note that $Y$ is $g$-orthogonal to $X, JX, N, JN$. In particular, $Y \in \Dfour$. We can identify $N$ with $U$; $X$ with $\chi$ and $Y$ with $-\Psi X$ in the contact frame of \Cite{liefsoens2024}. Then, we see that $g( [X,Y],N ) = -1$. This means that $\mathcal{D} = \set{N}^{\perp_g}$ is never involutive, so that $N$ cannot be the normal to a hypersurface.
\end{proof}

We proceed by giving orthonormal frames for both horizontal and non-horizontal hypersurfaces in the nearly Kähler $\CP^3$.
\begin{lemma}\label{thm:frame_horizontal}
    Suppose $\mathcal{H}$ is a horizontal hypersurface of nearly Kähler $\CP^3$ with unit normal $N$. Let $G = \nabla J$, and $\Phi, \Psi = J \Phi $ be the complex almost contact structures of $\CP^3$. Then, the frame $\set{e_1, \hdots, e_5}$ defined by
    \begin{align}
        e_1 &= -G(\Phi N,N) & e_2 &= J e_1 & e_3 &= JN &  e_4 &= \Phi N &  e_5 &= \Psi N
    \end{align}
    forms a $g$-orthonormal frame of $\mathcal{H}$. Moreover, $e_1, e_2 \in \Dtwo$ and $e_3, e_4, e_5 \in \Dfour$. 
\end{lemma}

\begin{lemma}\label{thm:frame_non_horizontal}
    Suppose $\mathcal{H}$ is a non-horizontal hypersurface of nearly Kähler $\CP^3$ with unit normal $N$. Let $G = \nabla J$, $P = -J \Jo$ be the almost product structure and $\mathcal{P}_1 = (\identity-P)/2$, $\mathcal{P}_2 = (\identity+P)/2$ the projections to $\Dtwo$ and $\Dfour$, respectively. Then, the frame $\set{e_1, \hdots, e_5}$ defined by
    \begingroup
    \allowdisplaybreaks
    \begin{align*}
        e_1 &= \frac{JN + \Jo N}{2 \cos \theta} = \frac{1}{\cos\theta} \mathcal{P}_2 J N &
        e_2 &= \frac{JN - \Jo N}{2 \sin \theta} = \frac{1}{\sin\theta} \mathcal{P}_1 J N \\
        e_3 &= \frac{PN - \cos(2\theta) N}{\sin(2\theta)} \\
        e_4 &= \frac{ G(PN, N) }{\sin(2\theta)} = \frac{ G(e_3, N) }{\sin(2\theta)} &
        e_5 &= J e_4 =  - \frac{ G(\Jo N, N) }{\sin(2\theta)}
    \end{align*}
    \endgroup
    forms a $g$-orthonormal frame of $\mathcal{H}$. 

    Moreover, the derivatives of $\theta$ are given by $\alpha(e_i, e_3) = e_i(\theta) + \delta_{i 5} \frac{1}{2}$, where $\alpha(x,y)=g(h(x,y),N)$ and $h$ is the second fundamental form.
\end{lemma}
\begin{proof}
    It is a simple exercise to show that the frame fields are orthonormal and tangent to $\mathcal{H}$. Since we know $\nabla \Jo$ and $G = \nabla J$ (and hence $\nabla P = - \nabla (J \Jo)$); and by imposing that nearly Kähler $\CP^3$ is of constant type, we know that $\nabla G(X,Y,Z) = g(Z,X)JY - g(Y,X)JZ -g(JY,Z)X$, from which the conditions on the derivatives of $\theta$ follow. 
\end{proof}
\begin{remark}\label{remark:orthogonality_for_all_ga}
    The frames of \Cref{thm:frame_horizontal} and \Cref{thm:frame_non_horizontal} are also orthogonal for all $g_a$. 
\end{remark}

With these frames, we can characterise a horizontal hypersurface in the following ways.
\begin{proposition}\label{thm:characterisation_horizontal}
    Suppose $\mathcal{H}$ is a hypersurface of nearly Kähler $\CP^3$. Then the following are equivalent
    \begin{itemize}
        \item $\mathcal{H}$ is horizontal,
        \item the distribution $\Span\set{N, JN, \Jo N, PN}$ is two dimensional,
        \item $\mathcal{H}$ is $P$-invariant, i.e. the distribution $\Span\set{JN}\cap T \mathcal{H}$ is preserved by $P$,
        \item $\mathcal{H}$ is the pre-image of a hypersurface in $\sphere{4}$.
    \end{itemize}
\end{proposition}
\begin{proof}
    Suppose $\mathcal{H}$ is horizontal, i.e. $\theta=0$. Then, the second and third points are obvious. For the fourth, consider the frame of \Cref{thm:frame_horizontal}. 
    Note that
    $e_1, e_2$ span the vertical space of the twistor fibration. Therefore, $\set{\dd \tau(JN), \dd \tau(\Phi N), \dd \tau(\Psi N), \dd \tau(N)}$ forms an orthonormal frame of $\sphere{4}$. Note that
    $\dd \tau( \nabla_X Y ) = \nabla^{\sphere{4}}_X Y$, since $\tau$ is a Riemannian submersion. 
    Then, we can check that 
     $\set{\dd \tau(JN), \dd \tau(\Phi N), \dd \tau(\Psi N)}$ forms an involutive distribution in $\sphere{4}$. Hence, by the Frobenius theorem, this distribution is integrable and the leafs are hypersurfaces in $\sphere{4}$. 

    Now we show that the three last points imply that $\mathcal{H}$ is horizontal. 
    First, assume the distribution $\Span\set{N, JN, \Jo N, PN}$ is two-dimensional. Since $N$ and $JN$ are orthogonal, and also $N$ and $\Jo N$ are orthogonal, we have $JN = \pm \Jo N$ or $N = \pm P N$. By \Cref{thm:no_vertical_hypersurfaces}, we know that then $N = PN$ and $\mathcal{H}$ is horizontal by \Cref{thm:characterisation_special_values}. 

    Second, assume $\mathcal{H}$ is $P$-invariant. If $\mathcal{H}$ is non-horizontal, we can work with the frame of \Cref{thm:frame_non_horizontal}, and consider $e_3$. We can compute that $g(e_3, JN)=0$ and $g(Pe_3, N) = \sin(2 \theta)$. Hence, $\mathcal{H}$ cannot be $P$-invariant, which is a contradiction. Hence, $\mathcal{H}$ has to be horizontal.

    Third, suppose $\mathcal{H}$ is the pre-image of a hypersurface $M$ in $\sphere{4}$. By taking a frame on $M$ inside $\sphere{4}$ and lifting this frame horizontally, we find that $N$ must be in $\Dfour$, which means that the pre-image is a horizontal hypersurface of nearly Kähler $\CP^3$.
\end{proof}

Finally, the frames allow to find the following relation between a normal of the hypersurface for $(\CP^3,g_a)$ and for the nearly Kähler $\CP^3$.
\begin{proposition}\label{thm:normal_ga_gNK}
    Let $\mathcal{H}$ be a hypersurface in $\CP^3$. 
    Let $N$ be the unit normal for the nearly Kähler metric $g$ and $N_a$ the unit normal for a metric $g_a$ on $\CP^3$. Then, up to sign, they are connected as follows
    \begin{align*}
        N_a &= \frac{1}{\sqrt{2+a +(2-a)\cos(2 \theta)}} \frac{1}{\sqrt{2a}}\l( (a+2)N - (a-2) PN \r) \\
        N &= \frac{\sqrt{2+a -(a-2)\cos(2 \theta)}}{4 \sqrt{2a}} ( (a+2) N_a + (a-2) P N_a ).
    \end{align*}
\end{proposition}
\begin{proof}
    Since $\set{e_1, \hdots, e_5}$ is orthogonal for $g_a$, the frame $\set{e^a_i}$ with $e^a_i = e_i / \norm{e_i}_a$ forms a frame for $\mathcal{H}$ that is $g_a$-orthonormal. Then, define $\tilde{N}_a = N - \sum_{i=1}^5 g_a(N, e^a_i)e^a_i$ and $N_a = \tilde{N}_a / \norm{\tilde{N}_a}_a$ and note that $N_a$ is a $g_a$-unit normal for $\mathcal{H}$. Now, the first formula is a simple computation in the frame. To obtain the second formula, use that $P^2 = \identity$.
\end{proof}

\subsection{Linking isotropic and constant angle hypersurfaces}
In the previous result, we linked the normal of a hypersurface according to the nearly Kähler metric to the normal according to another $g_a$ metric. With this, we can also link $g_a$-isotropic hypersurfaces with hypersurfaces with constant nearly Kähler angle. Indeed, given the nearly Kähler angle $\theta$, defined by $g(KN, N) = \cos 2 \theta$, we can retrieve the $g_a$ angle of the hypersurface via the following equation: 
\begin{equation}
     \cos(2 \theta_a) = g_a( K N_a, N_a ) = \frac{2 -a +(2+a) \cos (2 \theta )}{2+a +  (2-a) \cos (2 \theta ) }.
\end{equation}
Given that $\theta$ is constant and non-zero, we can set $a = 2 \cot^2 \theta >0$. Then, $g_a(KN_a, N_a) = 0$, so that the hypersurface is $g_a$-isotropic if $\theta>0$ is constant. Conversely, suppose that the hypersurface is $g_a$-isotropic. Then, we find that the corresponding nearly Kähler angle is given by 
\[ \frac{1}{2}\arccos\frac{a-2}{a+2}, \] and thus in particular constant. Note that this angle is strictly between $0$ and $\pi/2$. Moreover, as $a$ tends to $0$, the corresponding angle tends to $\pi/2$, and as $a$ grows unbounded, the angle tends to $0$. In other words, we just proved the following result.
\begin{proposition}\label{thm:link_ga_isotropic_constant_angle}
    Let $\mathcal{H}$ be a non-horizontal hypersurface in nearly Kähler $\CP^3$. Then, $\mathcal{H}$ has constant nearly Kähler angle $\theta$ if and only if $\mathcal{H}$ is $g_a$-isotropic for some $a$. In particular, the angle $\theta$ and the metric parameter $a$ are linked as
    \[ a = 2 \cot^2 \theta \quad \text{ and } \quad \theta = \frac{1}{2}\arccos\frac{a-2}{a+2}. \]
\end{proposition}
\begin{remark}
    Only for the horizontal case, $\theta = 0$, there is no other homogeneous metric linked to reduce to an isotropic case.
\end{remark}











\section{Examples of Hypersurfaces}\label{sec:examples}

We now give a family of examples of extrinsically homogeneous hypersurfaces that are Hopf in nearly Kähler $\CP^3$, and have many interesting properties.

Let $t \in (0, \pi/2)$ be constant and define $\psi_t$ by
\begin{equation}\label{eq:family_Hopf_hypersurfaces}
 	\psi_t: \sphere{3}\times\sphere{3} \to \sphere{7}: (p,q) \mapsto ( \cos t \, p, \sin t\, q ).
\end{equation} 
Denote $\lift{\mathcal{H}}_t = \psi_t(\sphere{3}\times\sphere{3})$ the hypersurface immersed in $\sphere{7}$, and let $\mathcal{H}_t = \pi(\lift{\mathcal{H}}_t)$ be the hypersurface in $\CP^3$.

\subsection{\texorpdfstring{$\mathcal{H}_t$}{ℋₜ} is extrinsically homogeneous}
First of all, note that $\lift{\mathcal{H}}_t$ is homogeneous. Indeed, we can identify $SU(2)$ with the unit quaternions $\sphere{3} \subset \quaternions$. Then, the group $\SU(2)\times \SU(2)$ acts transitively via the following action: $(A,B) \cdot (p,q) \coloneqq (Ap, Bq)$. Since $\lift{\mathcal{H}}_t$ is homogeneous, so is $\mathcal{H}_t$ via the action $(A,B) \cdot \pi(p,q) \coloneqq \pi(Ap, Bq)$. Moreover, note that the induced action of $\SU(2)\times \SU(2)$ on $\mathcal{H}_t$ is by $g_a$-isometries for all $a$, since $\mqty(A & 0 \\ 0 & B)$ is an element of $\Sp(2)$ for all unit quaternions $A,B$. For the isotropy, we choose the point $o = \pi\circ\psi_t(1,1)$, and we readily find that $(A,B)\cdot o = o$ if and only if $A=B = e^{i \alpha}$ for some real $\alpha$. We thus have \[ \mathcal{H}_t \cong \frac{\SU(2)\times \SU(2)}{ \Delta\U(1) } \] and that $\mathcal{H}_t$ is extrinsically homogeneous. Moreover, we thus find that $\mathcal{H}_t$ is diffeomorphic to $\sphere{2}\times \sphere{3}$ via the map $\mathcal{H}_t \to \sphere{2}\times \sphere{3}: (p,q)\Delta\U(1) \mapsto (p \Delta\U(1), q^{-1}p )$.

\subsection{\texorpdfstring{$\mathcal{H}_t$}{ℋₜ} is Hopf}
To prove that $\mathcal{H}_t$ is Hopf, we choose Hopf coordinates on $\sphere{3}$:
\begin{align*}
    p &= ( \cos x_2 \sin x_1 , \sin x_2 \sin x_1, \cos x_3 \cos x_1, \sin x_3 \cos x_1 ) \\
    q &= ( \cos y_2 \sin y_1 , \sin y_2 \sin y_1, \cos y_3 \cos y_1, \sin y_3 \cos y_1 ). 
\end{align*}

We have the coordinate frame
\begin{align*}
    f_{i} &= \pdv{\psi_t}{x_i} & & i\in \set{1,2,3} &
    f_{i} &= \pdv{\psi_t}{y_{i-3}} & &  i\in \set{4,5,6},
\end{align*} which is orthogonal with respect to the round metric on $\sphere{7}$. The normal (round metric) $\xi$ on $\lift{\mathcal{H}}_t$ is given by $\xi = (\sin t \, p, -\cos t \, q )$. Denote $\langle \cdot, \cdot\rangle$ the standard round metric on $\sphere{7}$, and $\sphere{3}$. By embedding $\sphere{3} \subset \quaternions $ and $\sphere{7} \subset \quaternions^2$, we find that $\langle \psi_t, \xi \rangle = \langle i \psi_t, \xi \rangle = \langle j \psi_t, \xi \rangle = \langle k \psi_t, \xi \rangle = 0$.
Hence, $\xi \in \lift{\Dfour}$.

Since $i \xi$ is tangent to $\lift{\mathcal{H}}_t$, we can decompose it with respect to the coordinate frame $\set{f_i}$. We have $i \xi = \tan t\, (f_2 + f_3) - \cot t\, (f_5 + f_6)$.
Hence, the shape operator $\lift{\FS{A}}$ associated to $\xi$ and the metric induced from the round metric acts as follows on $i \xi$
\begin{align*}
    \lift{\FS{A}}_\xi (i \xi) 
    &= - (\nabla_{i \xi} \xi)^\top 
    = - \tan t \, \l( \pdv{i \xi}{x_2} + \pdv{i \xi}{x_3} \r)^\top + \cot t \, \l( \pdv{i \xi}{y_2} + \pdv{i \xi}{y_3} \r)^\top \\
    &= (\cot t - \tan t) i \xi  + i \psi_t
    = 2\cot(2t) \, i \xi  + i \psi_t.
\end{align*}

Projecting this result to $\CP^3$, and writing $\FS{N} = \dd\pi \xi$, we have 
\begin{align*}
    \FS{A}_{\FS{N}} ( \Jo \FS{N}) 
    = 2 \cot(2t) \, \Jo \FS{N},
\end{align*} where $\FS{A}_{\FS{N}}$ is the shape operator with respect to $\FSm$. In other words, $\mathcal{H}_t=\pi(\lift{\mathcal{H}}_t)$ is a Hopf hypersurface in Kähler $\CP^3$.

Since $\FS{N} \in \Dfour$, we have $\Jo \FS{N} = J \FS{N}$ and that $N = \FS{N}/\sqrt{2}$ is a $g$-unit normal on $\mathcal{H}$. The shape operator $A_{\FS{N}}$ with respect to $g$ is related to $\FS{A}_{\FS{N}}$ as $A_{\FS{N}} X = \FS{A}_{\FS{N}} X - D(X,\FS{N})^\top$. Since $\FS{N}, \Jo \FS{N} \in \Dfour$, the difference tensor $D(\Jo \FS{N}, \FS{N})^\top$ vanishes. Hence, we obtain
\begin{align*}
    A_{N} ( J N)
    = \sqrt{2} \cot(2t) J N.
\end{align*}
As such, $\mathcal{H}_t$ is also a Hopf hypersurface in nearly Kähler $\CP^3$. Furthermore, with the same arguments, we find that these examples are also Hopf in $(\CP^3, g_a, J)$ and $(\CP^3, g_a, \Jo)$.

\subsection{\texorpdfstring{$\mathcal{H}_t$}{ℋₜ} as a tube}
We can also calculate the principal curvatures for the Fubini-Study metric of this example, and find $2 \cot(2t)$ with multiplicity one, and $-\tan t$ and $\cot t$, each with multiplicity two. From the classification in \Cite{takagi1975}, we then find that these examples are locally congruent to tubes over a $\FSm$-totally geodesic $\CP^1$.

\subsection{\texorpdfstring{$\mathcal{H}_t$}{ℋₜ} projects to geodesic \texorpdfstring{$\sphere{3} \subset \sphere{4}$}{𝕊³ ⊂ 𝕊⁴}}
Let $\tau$ denote the twistor fibration. We have the following mapping:
\[ \tau \circ \pi : \sphere{7} \subset \quaternions^2 \to \sphere{4}  \cong \quaternions P^1:  (p,q) \mapsto [(p,q)].   \]
We find that $\eval{\tau \circ \pi}_{\Im \psi_t }$ gives a hypersurface in $\sphere{4}$. Moreover, by studying this hypersurface further, we find that it is a geodesic $3$-sphere inside $\sphere{4}$. Moreover, it is a totally geodesic grand $3$-sphere if and only if $t = \pi/4$. This case coincides exactly with $JN$ being in the kernel of $A_N$. Finally, only in this case, $\mathcal{H}_t$ is minimal (in both Kähler and nearly Kähler $\CP^3$).

\subsection{\texorpdfstring{$\mathcal{H}_t$}{ℋₜ}: extrinsic and intrinsic differences }
Further on the mean curvature, we have that $H = \frac{1}{\sqrt{a}}2 \cot(2t) N$, so that $\mathcal{H}_t$ is extrinsically different for all $t$, for all $a$. The non-normalised scalar curvature, however, is given by 
$\frac{2 \left(2 \sqrt{2} a^{3/2}+6 a^3+a^2 \left(6 \lambda ^2+20\right)-21 a-2 \sqrt{2} \sqrt{a}+18\right)}{a^3}$, 
 and so this only says that for all $\pi/2> t \geq \pi/4$, $\mathcal{H}_t$ is intrinsically different. Interestingly, however, there is an intrinsic mirror symmetry around $t = \pi/4$ that we will discuss here.

The pull-back metric $g^t_{(p,q)}$ on $\SU(2)\times \SU(2) \cong \sphere{3}\times \sphere{3}$ from the immersion of $\mathcal{H}_t$ into the squashed $7$-sphere is given by
\begin{align*}
    g^t_{(p,q)}((v_1,w_1),(v_2,w_2) ) &= a \l( \cos^2 t \, \langle v_1, v_2\rangle_p + \sin^2 t \, \langle w_1, w_2\rangle_q \r) \\ 
    &+ (1-a) \l(
    \cos^2 t \langle v_1, a_1 j v_2 + a_2 k v_2  \rangle_p +
    \sin^2 t \langle w_1, a_1 j w_2 + a_2 k w_2  \rangle_q \r),
\end{align*}
where
    $a_1 =  \cos^2 t \langle v_2, jp \rangle + \sin^2 t \langle w_2, jq \rangle$,
    $a_2 =  \cos^2 t \langle v_2, k p \rangle + \sin^2 t \langle w_2,  k q \rangle$,
and $\langle \cdot, \cdot \rangle$ denotes the standard inner product on $\sphere{3}$. Note that all these metrics are non-degenerate, since $t \in (0, \pi/2)$. Moreover, it takes a simple computation to show that the $\Delta\U(1)$-action of the isotropy leaves these metrics invariant. Hence, the metric descends to $\frac{\SU(2)\times \SU(2)}{\Delta\U(1)}$, and we will denote this metric with the same symbol. With this metric on the quotient, we get that \[ \l( \frac{\SU(2)\times \SU(2)}{\Delta\U(1)}, g^t \r) \to \l( \mathcal{H}_t, g \r): (A,B)\Delta\U(1) \mapsto (\pi\circ \psi_t) (A,B), \] is an isometry, where $g$ denotes the induced metric from $\CP^3$ to $\mathcal{H}_t$. 

Consider for all $s \in (0, \pi/4)$ the map \[ \rho_s : \l( \frac{\SU(2)\times \SU(2)}{\Delta\U(1)}, g^{s_-} \r) \to \l( \frac{\SU(2)\times \SU(2)}{\Delta\U(1)}, g^{s_+} \r): (p,q)\Delta\U(1) \mapsto (q,p)\Delta\U(1), \] where $s_\pm = \frac{\pi}{4}\pm s$. 
Since $\sin^2\l( s_- \r) = \cos^2\l( s_+ \r)$ and $\cos^2\l( s_- \r) = \sin^2\l( s_+ \r)$, we see that $\rho_s$ is an isometry. Hence, $\mathcal{H}_{s_+}$ and $\mathcal{H}_{s_-}$ are intrinsically the same, but extrinsically different.

\begin{remark}
    Given the immersion of \Cref{eq:family_Hopf_hypersurfaces_intro}, we can act with an element of the isometry group of Kähler $\CP^3$, which is not an isometry of nearly Kähler $\CP^3$, to obtain a non-congruent hypersurface, which will still be Hopf for Kähler $\CP^3$. It takes a straightforward computation to show that the angle for all of these hypersurfaces is not constant.
\end{remark}

\section{Hopf hypersurfaces in the nearly Kähler \texorpdfstring{$\CP^3$}{ℂP³}}\label{sec:hopf_hypersurfaces}

In this section, we classify the Hopf hypersurfaces of nearly Kähler $\CP^3$. We start by excluding constant angle Hopf hypersurfaces that are not horizontal. 

\begin{proposition}\label{thm:hopf_cst_angle_is_horizontal}
  Suppose $\mathcal{H}$ is a Hopf hypersurface of nearly Kähler $\CP^3$ with constant angle $\theta$, then $\mathcal{H}$ must be horizontal.
\end{proposition}
\begin{proof}
    Suppose $\mathcal{H}$ is non-horizontal, so that we can work in the frame of \Cref{thm:frame_non_horizontal}. In particular then, we have the following components of the second fundamental form: $2 \alpha_{i3} = \delta_{i5}$. Demanding that $\mathcal{H}$ is Hopf, gives for $i \in \set{3,4,5}$ the following:
    $\alpha_{11} = \lambda -\alpha_{12} \tan\theta$, $\alpha_{22} = \lambda  - \alpha_{12} \cot\theta$ and $\alpha_{1i} = - \alpha_{2i} \tan\theta$.
    We perform the following rotation in the frame: $\tilde{e}_1 = \cos \theta e_1 + \sin \theta e_2$, $\tilde{e}_2 = -\sin \theta e_1 + \cos \theta e_2$ and note that $JN = \tilde{e}_1$.
    From the Codazzi equation applied to $(e_i, e_1, e_1)$ for $i=3, 4, 5$ and $(\tilde{e}_2, e_1, e_1)$, we get
    \begin{align}\label{eq:derivatives_lambda}
        e_2(\lambda) &= \tan\theta e_1(\lambda) &
        e_3(\lambda) &= \frac{1}{2} \sin(2\theta) (5\cos(2\theta) -1) & 
        e_4(\lambda) &= e_5(\lambda) = 0.
    \end{align}
    We then consider the compatibility conditions $c_{ij}$ for $\lambda$, i.e. $c_{ij}(\lambda) = [e_i, e_j]\lambda - ( e_i(e_j(\lambda)) - e_j(e_i(\lambda)) )$.
    From $c_{13}(\lambda)$, we get $e_3(e_1(\lambda)) + \alpha_{12} e_1(\lambda) \sec^2\theta - \lambda e_1(\lambda) \tan\theta=0$. Combining this with $c_{23}(\lambda)$, we get
    $e_1(\lambda) ( \sin(2 \theta) \lambda - 2\alpha_{12} ) = 0$.
    If $e_1(\lambda) \neq 0$, we get a contradiction with the other compatibility conditions. So, we have $e_1(\lambda)=0$. Then, from $c_{35}(\lambda)$, we get $2 \theta = \arccos(1/5)$. From the Codazzi-equation applied to $(e_4, \tilde{e}_2, \tilde{e}_1), (e_4, e_5, \tilde{e}_1), (e_5, \tilde{e}_2, \tilde{e}_1)$, we find $\alpha_{25}=0$, and $\alpha_{12},\alpha_{55}$ in terms of $\lambda$. However, from the Codazzi equation applied to $(\tilde{e}_2, e_4, e_5)$, we then get that $2+5 \lambda^2 = 0$, which has no solution for real $\lambda$. Hence, $\mathcal{H}$ cannot be non-horizontal, Hopf and have constant angle. 
\end{proof}

Knowing \Cref{thm:hopf_cst_angle_is_horizontal}, we can prove that every Hopf hypersurface has to be horizontal.

\begin{proposition}\label{thm:hopf_is_horizontal}
    Suppose $\mathcal{H}$ is a Hopf hypersurface of nearly Kähler $\CP^3$, then $\mathcal{H}$ must be horizontal.
\end{proposition}
\begin{proof}
    We prove this via contradiction, and work in the frame of \Cref{thm:frame_non_horizontal}.
    Just as in \Cref{thm:hopf_cst_angle_is_horizontal}, \Cref{eq:derivatives_lambda} holds. Again, we consider the compatibility conditions $c_{ij}$ for $\lambda$ as introduced in the proof of \Cref{thm:hopf_cst_angle_is_horizontal}.
    From $c_{45}(\lambda)$, we get that either $e_1(\lambda) = 0$, or that $\alpha_{55} = - \alpha_{44}$. The former, together with $c_{14}(\lambda)$ implies that $\theta$ is constant, which cannot be the case, due to \Cref{thm:hopf_cst_angle_is_horizontal}. Hence, we know that $\alpha_{55} = - \alpha_{44}$. Then, combining the Codazzi equation applied to $(e_4, JN, e_4)$ and $(e_5, JN, e_5)$, we get $\alpha_{34} = -\alpha_{25} \sec\theta $. Then, $c_{23}(\lambda), c_{23}(\lambda), c_{25}(\lambda), c_{34}(\lambda), c_{24}(\lambda)$ and $c_{14}(\lambda)$ gives $\alpha_{33}$, $\alpha_{35}$, $\alpha_{25}$ and $\alpha_{24}$ in terms of $e_1(\lambda)$ and $\theta$ (and $\alpha_{12}$, $\alpha_{13}$ for $\alpha_{33}$). Then, we find that $c_{35}(\lambda)$ gives $e_1(\lambda)^2 = f(\theta)^2$ for some function that only depends on $\theta$. Of course, only $e_1(\lambda) = \pm f(\theta)$ satisfies this equation. We will now show that this solution is not possible.

    Differentiating $e_1(\lambda)^2 = f(\theta)^2$ with respect to $e_5$, we get $2 e_1(\lambda) e_5(e_1(\lambda)) = 2 f(\theta) f'(\theta) e_5(\theta)$.
    Combining this equation with the condition of $c_{15}(\lambda)$, we find that $e_1(\lambda) = \pm f(\theta)$ is only a solution for specific constant values of $\theta$. Now we are done, thanks to \Cref{thm:hopf_cst_angle_is_horizontal}.
\end{proof}

With the previous two results, we are ready to classify all Hopf hypersurfaces in nearly Kähler $\CP^3$. 

\begin{theorem}\label{thm:classification_Hopf}
	Suppose $\mathcal{H}$ is a Hopf hypersurface of nearly Kähler $\CP^3$, then $\mathcal{H}$ is locally congruent to a member of the family of \Cref{eq:family_Hopf_hypersurfaces_intro}.
\end{theorem}
\begin{proof}
	By \Cref{thm:hopf_is_horizontal}, we only need to consider $\theta = 0$. We use the frame of \Cref{thm:frame_horizontal}. Assume that $\mathcal{H}$ is Hopf, so that $A_N JN = \lambda JN /\sqrt{2}$ for some function $\lambda$. The factor $\sqrt{2}$ is convenient later on. We have the following conditions on the second fundamental form: $\alpha_{33} =  \lambda/\sqrt{2}$, $\alpha_{34}=0$, and $\alpha_{35}=0$. Then, we consider the Codazzi equation for $(e_i, e_j, e_k)$ with $(i,j,k)=(3,1,4),(3,1,5),(3,2,5)$, to find that $\alpha_{44}=\alpha_{55} = \lambda/\sqrt{2}$ and $\alpha_{45}=0$. The remaining Codazzi equations then demand that $\lambda$ is a constant. 

    The second fundamental form is now completely known in terms of $\lambda$. We now transition from the nearly Kähler metric $g$ to the Kähler metric $\FSm$. We then compute the principal curvatures of the $\FSm$-second fundamental form, and find three eigenvalues: 
    $\lambda$, $\frac{1}{2} \l( \lambda + \sqrt{4+\lambda^2} \r)$ and $\frac{1}{2} \l( \lambda - \sqrt{4+\lambda^2} \r)$.
    Here, the last two have multiplicity two, and the first one has multiplicity one. By Theorem T of \Cite{kimura1989} (summarised from \Cite{takagi1973}), the hypersurface is locally congruent to either a tube over a totally geodesic $\CP^1$ or over a complex quadric. In any case, $\lambda = 2 \cot(2t)$. Then, the other principal curvatures become $-\tan t$ and $\cot t$, so that the hypersurface is locally congruent to a tube over a totally geodesic $\CP^1$ by the table in \Cite{takagi1975}. By the considerations of \Cref{sec:examples}, the hypersurface $\mathcal{H}$ is then locally congruent to a member of the family \Cref{eq:family_Hopf_hypersurfaces}, where $\lambda = 2 \cot(2t)$, after projection with $\pi$. 
\end{proof}

We extend the result of Hopf hypersurfaces of nearly Kähler $\CP^3$ as follows.

\begin{proposition}\label{thm:horizontal_hopf_for_both}
    Let $M$ be a horizontal hypersurface of $\CP^3$  and let $a \neq 1$. Suppose $M$ is Hopf for $(\CP^3, g_a, J)$ or for $(\CP^3, g_a, \Jo)$, then $M$ is locally congruent to a member of \Cref{eq:family_Hopf_hypersurfaces_intro}. 
\end{proposition}
\begin{proof}
    To prove this, we show that the assumption automatically leads to the hypersurface being Hopf for nearly Kähler $\CP^3$. Then, we invoke \Cref{thm:classification_Hopf} and are done.

    With the assumption that $M$ is horizontal, we know that for all $a$, the normal $N_a$ lies in $\Dfour$. Hence, $J N_a = \Jo N_a$. Moreover, for all $a$ and $b$, the normals $N_a$ and $N_b$ are proportional to each other. It follows that under the assumptions, the hypersurface is automatically also Hopf for nearly Kähler $\CP^3$ and this finishes the proof. 
\end{proof}

\begin{theorem}
    Let $M$ be a hypersurface of $\CP^3$. Suppose $M$ is Hopf for $(\CP^3, g_a, J)$ and $(\CP^3, g_a, \Jo)$ for any $a>0$, then $M$ is locally congruent to a member of \Cref{eq:family_Hopf_hypersurfaces_intro}. In particular, the family of examples of \Cref{eq:family_Hopf_hypersurfaces_intro} is unique among the Hopf hypersurfaces of Kähler $(\CP^3, g_1, \Jo)$, in that it is also Hopf for $(\CP^3, g_1, J)$ (with symmetry group reduced so that $J$ is well defined). 
\end{theorem}
\begin{proof}
    Due to \Cref{thm:horizontal_hopf_for_both}, we only need to show that there is no such hypersurface when we also assume that it is non-horizontal. Demanding that it is Hopf for $(\CP^3, g_a, J)$ with parameter $\lambda$ gives expressions for $\alpha_{11}, \alpha_{22}, \alpha_{13}, \alpha_{14}, \alpha_{15}$. Additionally demanding that it is Hopf for $(\CP^3, g_a, \Jo)$ with parameter $\mu$  gives expressions for $\alpha_{23}, \alpha_{24}, \alpha_{25}$ and $\alpha_{12}$, as well as the following equation: $(a - 2 \cot^2\theta)(\lambda- \mu) = 0$. Hence, we have two cases to consider: either $\lambda$ and $\mu$ are equal, or $\theta$ is constant and the hypersurface is $g_a$-isotropic according to \Cref{thm:link_ga_isotropic_constant_angle}.

    In the first case, it follows from the Codazzi applied to $(e_3,e_5,e_2)$ that $\alpha_{45}=0$. Combining the Codazzi equations for $(e_4,e_5,e_2), (e_2,e_5,e_3), (e_2,e_5,e_4), (e_1,e_3,e_1)$, we get an expression for $\alpha_{35}$ in terms of $a$ and $\theta$, and that $\alpha_{34}=0$. The Codazzi equation applied to $(e_1,e_3,e_2)$ then dictates that either $\theta$ is constant (which turns out to be impossible with the second fundamental form components calculated thus far) or that $a=1/2$. However, the Codazzi equations for $(e_2,e_3,e_4), (e_1,e_4,e_5)$ give that $\alpha_{44}=0$ and that $\cos \theta = 0$, which is a contradiction. 

    Hence, we need to have $a = 2 \cot^2 \theta$ and $\lambda\neq \mu$. In particular, $\theta$ is constant, from which it follows that $\alpha_{33}=\alpha_{34}=\alpha_{35}-1/2 = 0$. Furthermore, we also have $\alpha_{45} = 0$ from the Codazzi equation applied to $(e_3,e_5,e_2)$. From the Codazzi equation applied to $(e_4,e_5,e_2), (e_2,e_3,e_4), (e_1,e_3,e_4)$, we get expressions for $\alpha_{44}$ and $\alpha_{55}$, as well as $\mu = \frac{1}{8}\lambda (3 - \cos(2\theta) + 6 \sec^2 \theta )$. Since $\lambda \neq \mu$, we then also have that $\lambda \neq 0$. Now all components $\alpha_{ij}$ are expressed in terms of $\theta$ and $\lambda$. The system of equations obtained from applying the Codazzi equation to $(e_1,e_2,e_5)$ and $(e_4,e_1,e_5)$ is a linear system in $X = \lambda^2$. This has a unique solution for $X$, but it is negative, so that there are no real solutions for $\lambda$. Hence, we get a contradiction also in the final case, and we are done.
\end{proof}

\section{Extrinsically homogeneous hypersurfaces}\label{sec:extr_homo}

\begin{theorem}
    Let $a \neq 1$, then all extrinsically homogeneous hypersurfaces of $(\CP^3,g_a)$ are locally congruent to the examples of \Cref{eq:family_Hopf_hypersurfaces_intro}. In particular, they are all Hopf and horizontal. 
\end{theorem}
\begin{proof}
    Since the isometry group for all $(\CP^3,g_a)$ ($a \neq 1$) is the same, we can consider all extrinsically homogeneous hypersurfaces of $(\CP^3,g_a)$ and only study them from the perspective of nearly Kähler $\CP^3$. Hence, let $\mathcal{H}$ be an extrinsically homogeneous hypersurface with $g$-unit normal $N$. Let $H$ be the closed subgroup of the isometry group of nearly Kähler $\CP^3$ such that $\mathcal{H}$ is the orbit of $H$. Let $\phii \in H$ be any isometry. Note that $\phii$ is orientation preserving, per definition. Hence, $\dd\phii J = J \dd\phii$. 
    For all isometries $\phii$ of NK $\CP^3$ we also have $P \dd\phii = \dd \phii P$, see \Cite{liefsoens2024}.
    We consider two cases.

    \textit{$\mathcal{H}$ is non-horizontal.}
    Suppose $\mathcal{H}$ is non-horizontal, and let $\tilde{N} = \dd \phii N$. For all $p\in \mathcal{H}$, we have
    \begin{align*}
        \cos\theta(\phii(p)) &= \cos\theta(\phii(p))\, \norm{ \dd\phii e_1}_{\phii(p)}
        = \norm{ \dd\phii \cos\theta(p) e_1}_{\phii(p)} \\
        &= \norm{ \dd\phii \mathcal{P}_2 J N}_{\phii(p)} 
        = \norm{ \mathcal{P}_2 J N}_{p} 
        = \cos\theta(p).
    \end{align*}
    Similarly, we have $\sin\theta(\phii(p)) = \sin\theta(p)$.
    We conclude that $\cos (\theta \circ \phii) = \cos \theta$ and $\sin (\theta \circ \phii) = \sin\theta$, for all $\phii \in H$. Hence, on each connected component of $\mathcal{H}$, $\theta$ is constant. 

    Now, fix some arbitrary $\phii \in H$, and define $\tilde{e}_i = \dd\phii e_i$. Note that since $\phii$ is an isometry, and $J$ is preserved under $\phii$, also $G = \nabla J$ is preserved. Hence, since also $\phii P = P \phii$; $\phii J = J \phii$ and $\phii \Jo = \Jo \phii$, we find that the frame $\tilde{e}_i$ is simply the frame $e_i$ with $N$ replaced by $\dd\phii N$. 

    Take $p,q \in \mathcal{H}$. Take an isometry $\phii \in H$ such that $\phii(p) = q$. 
    Then, we calculate
    \begin{align*}
        \alpha_{ij} (q) 
        &= \alpha_{ij} (\phii(p)) 
        = g_{\phii(p)} \l( \eval{\nabla_{ e_i (\phii(p)) } e_j}_{\phii(p)} , \dd\phii_p N_p \r) \\
        &= g_{\phii(p)} \l( \eval{\nabla_{ \dd \phii_p e_i (p) } \dd \phii_p e_j}_{\phii(p)} , \dd\phii_p N_p \r) 
        = g_{p} \l( \eval{\nabla_{ e_i (p) } e_j}_{p} , N_p \r) 
        = \alpha_{ij}(p).
    \end{align*}
    As such, all functions $\alpha_{ij}$ are constant on $\mathcal{H}$. 

    Then, we consider the Codazzi equation applied to $(e_2,e_1,e_2)$ and $(e_5,e_3,e_5)$ to obtain the equations $\alpha_{25} = 0$ and $1 + \alpha_{15}^2 + \alpha_{45}^2 +\alpha_{55}^2 +\cos^2\theta - \alpha_{25}^2 \cot^2 \theta = 0$. Hence, the Codazzi equation is never satisfied, and an extrinsically homogeneous manifold cannot be non-horizontal.

    \textit{$\mathcal{H}$ is horizontal.}
    With the same argument as in the previous case, we can prove that if $\mathcal{H}$ is an extrinsically homogeneous horizontal hypersurface, that then $g(h(JN, JN),N)$ is constant on the manifold. 
    From the Codazzi equation applied to $(e_3, e_1, e_3)$; $(e_3, e_2, e_3)$; $(e_3, e_1, e_5)$; $(e_3, e_1, e_4)$; and $(e_3, e_2, e_5)$, we get the following system of equations: $4 \alpha_{34} = e_1(\alpha_{33})$; $4 \alpha_{35} = -e_2(\alpha_{33})$; $2 \alpha_{45}  = e_1(\alpha_{34}) + \alpha_{34} \, g( \nabla_{e_1} e_4, e_5)$; $-2 \alpha_{33} + 2 \alpha_{44}  = e_1(\alpha_{34}) - \alpha_{35} \, g( \nabla_{e_1} e_4, e_5)$; and $2 \alpha_{33} - 2 \alpha_{55}  = e_2(\alpha_{35}) + \alpha_{34} \, g( \nabla_{e_2} e_4, e_5)$.
    Since $g(h(JN, JN),N) = \alpha_{33}$ is a constant, we find $\alpha_{34}=\alpha_{35} = 0$ so that $\alpha_{45} = 0$, and $\alpha_{33} = \alpha_{44} = \alpha_{55}$. It follows that $A_N JN = \alpha_{33} JN$, so that $\mathcal{H}$ is Hopf. The proof is finished by invoking \Cref{thm:classification_Hopf}.
\end{proof}

\section{Proof of \texorpdfstring{\Cref{thm:main_theorem_NK_CP3}}{Theorem A} }
    In this section, we finish the proof of \Cref{thm:main_theorem_NK_CP3}. Together with the results of \Cref{sec:extr_homo,sec:hopf_hypersurfaces}, the only things left to show is that $\phi A + A \phi = 0$ implies $\mathcal{H}$ is Hopf with $A_N JN = 0$; and that a horizontal hypersurface with constant principal curvature has to be Hopf.

    For the former, suppose $\phi A + A \phi = 0$. Then, for $X = A_N JN$, we get $0 = (JX)^\top$. Hence, $JX = -b N$ for some $b$. In other words, $X = b J N$, or $\mathcal{H}$ is Hopf. By \Cref{thm:classification_Hopf}, we can easily check that in this case $b= 0$. 

    For the latter, take the frame of \Cref{thm:frame_horizontal} and let again $\alpha_{ij} = g(h(e_i, e_j),N)$. For a horizontal hypersurface with constant principal curvature, we can take the determinant of the second fundamental form, to obtain $\alpha_{33}/16$. Since the principal curvatures are constant, so is their product, and hence $\alpha_{33}$ is constant. From the Codazzi equation applied to $(e_1, e_3, e_3)$ and $(e_2, e_3, e_3)$, we find $\alpha_{34}=\alpha_{35}=0$, so that $\mathcal{H}$ is Hopf.

\section{Non-existence results}\label{sec:non_existence}

As stated in the introduction, hypersurfaces of nearly Kähler $\CP^3$ have been studied in \Cite{deschamps2021}, where they proved that totally geodesic hypersurfaces and totally umbilical hypersurfaces do not exist. In this section, we extend this by proving that Codazzi hypersurfaces do not exist for all homogeneous metrics on $\CP^3$. Furthermore, we also extend their result by proving that totally umbilical hypersurfaces are not possible in any $(\CP^3,g_a)$. 

\begin{theorem}
    There exists no Codazzi hypersurface in $(\CP^3,g_a)$. In particular, there are no parallel or totally geodesic hypersurfaces.
\end{theorem}
\begin{proof}
    \textit{The hypersurface has to be horizontal.}
    \sloppy From the Codazzi-equation, we see that $g_a(R^a(X,Y)Z, N_a)$ has to vanish identically for all tangent vector fields $X,Y,Z$. Take $e_1, e_4, e_5$ from \Cref{thm:frame_non_horizontal}, we find 
    \[ 0 = g_a(R^a(e_1,e_4)e_5, N_a) = \frac{\cos \theta}{\sqrt{2a}\sqrt{2+a-(a-2)\cos(2 \theta)}}. \] Hence, if it exists, it has to be horizontal. 

    \textit{The hypersurface cannot be horizontal.}
    We consider again the Codazzi-equation, this time for $e_3, e_4, e_5$ from \Cref{thm:frame_horizontal}. We have $g_a(R^a(e_3,e_4)e_5, N_a) = -\frac{1}{2 \sqrt{2a}}$, which can never be zero. This concludes the proof.
\end{proof}

\begin{theorem}\label{thm:TU}
    There exists no totally umbilical hypersurface in $(\CP^3,g_a)$.
\end{theorem}
\begin{proof}
    \textit{The hypersurface has to be horizontal.}
    Recall from \Cref{remark:orthogonality_for_all_ga} that the frame of \Cref{thm:frame_non_horizontal} is $g_a$ orthogonal. We denote the renormalised frame by $e^a_i$. From the equation $h^a(X,Y) = g_a(X,Y) H_a$, we find that all components of the second fundamental form can be expressed in terms of one function $f$. From the Gauss equation applied to $e^a_1, e^a_5, e^a_3, e^a_4$, we find that $\sin \theta$ has to vanish. Hence, a totally umbilical hypersurface has to be horizontal.
    
    \textit{The hypersurface cannot be horizontal.}
    Consider $h^a(e_1, e_5) = \frac{2-2a+\sqrt{2a}a}{\sqrt{2a}a}$. Since $e_1, e_5$ are $g_a$ orthogonal, this has to be zero. However, there are no real solutions for $a$ that satisfy $h^a(e_1, e_5) = 0$ and hence, the hypersurface cannot be totally umbilical. This concludes the proof.
\end{proof}

Finally, we remark that there are also no hypersurfaces in nearly Kähler $\CP^3$ with constant sectional curvature. This has been proved recently in a broader context of the six-dimensional homogeneous nearly Kähler manifolds and will be published soon.
We extend the result for nearly Kähler $\CP^3$ to all homogeneous $(\CP^3,g_a)$.


\begin{theorem}
    There are no constant sectional curvature hypersurfaces in $(\CP^3,g_a)$.
\end{theorem}
\begin{proof}
    
    We start from the Codazzi-equation, i.e. \Cref{eq:Codazzi}, and differentiate this equation with respect to $U$. Since $\nabla^a_U N_a$ is tangent to the hypersurface; by taking the cyclic sum and by invoking the Ricci-identity, we obtain
    \begingroup
    \allowdisplaybreaks
    \begin{align*}
        & \mathfrak{S}_{U,X,Y} \Big(  g( (\nabla R)(U,X,Y,Z) , N_a) + \alpha(U,X) g( R(N_a,Y)Z , N_a) \\ &
        + \alpha(U,Y)  g( R(X,N_a)Z , N_a)+ \alpha(U,Z)  g( R(X,Y)N_a , N_a) \Big)
        \\ & + g( R(X,Y)Z , \nabla_U N_a) \\
        &= \mathfrak{S}_{U,X,Y} \Big( g( (\nabla^2 h)(U,X,Y,Z) - (\nabla^2 h)(X,U, Y,Z), N_a)  \Big) \\
        &= \mathfrak{S}_{U,X,Y} \Big( \alpha(Y,Z) g( R^\perp(U,X)N_a, N_a) - \alpha(R^H(U,X)Y, Z) - \alpha(Y, R^H(U,X)Z)    \Big).
    \end{align*}
    \endgroup
    The assumption of constant sectional curvature means that $R^H(X,Y)Z = c (X \wedge_a Y)Z$ for some constant $c$. Finally, because of the symmetries of the curvature tensor, we obtain 
    \begin{align}
        & \mathfrak{S}_{U,X,Y} \Big(  g( (\nabla R)(U,X,Y,Z) , N_a) + \alpha(U,X) g( R(N_a,Y)Z , N_a) \nonumber \\ &
        + \alpha(U,Y)  g( R(X,N_a)Z , N_a) \Big)  + g( R(X,Y)Z , \nabla_U N_a) \nonumber  \\
        &= c \, \mathfrak{S}_{U,X,Y} \Big( - \alpha((U \wedge_a X)Y, Z) - \alpha(Y, (U \wedge_a X)Z)    \Big).
    \end{align}
    By using the symmetries of $\alpha$ and $g_a$ together with the cyclic sum, we find that the right-hand-side vanishes identically. Hence, we have for constant sectional curvature hypersurfaces the following necessary condition:
    \begin{align} \label{eq:csc_ricci_identities}
        & \mathfrak{S}_{U,X,Y} \Big(  g( (\nabla R)(U,X,Y,Z) , N_a) + \alpha(U,X) g( R(N_a,Y)Z , N_a) \nonumber \\ &
        + \alpha(U,Y)  g( R(X,N_a)Z , N_a) \Big)  + g( R(X,Y)Z , \nabla_U N_a)  \Big)  =0
    \end{align}

    We now consider horizontal and non-horizontal hypersurfaces separately.

    \textit{The hypersurface cannot be horizontal.}
    Consider the frame of \Cref{thm:frame_horizontal}.
    Applying the above equation to $(U,X,Y,Z) = (e^a_3,e^a_4,e^a_2,e^a_3)$, we find that $a=1$ or $a=2/9$. However, neither of these solutions is compatible with \Cref{eq:csc_ricci_identities} for $(U,X,Y,Z) = (e^a_2,e^a_5,e^a_1,e^a_2)$. Hence, no horizontal hypersurface can have constant sectional surface.

    \textit{The hypersurface cannot be non-horizontal.}
    We now work in the frame of \Cref{thm:frame_non_horizontal}. If $2a\neq 1$, and $a \neq 1$, we can solve \Cref{eq:csc_ricci_identities} completely, and find that the hypersurface has to be totally umbilical. However, by \Cref{thm:TU}, this is not possible. For the case $a=1/2$, \Cref{eq:csc_ricci_identities} gives that $h^a(e^a_i,e^a_j) = \delta_{ij} ( f_1 + f_2 \delta_{i2} )$, for some functions $f_1,f_2$. Finally, from the Gauss equation applied to $e^a_1, e^a_2, e^a_3, e^a_5$, we find that $f_2 = 0$, so that the hypersurface has to be totally umbilical. Again, by \Cref{thm:TU}, this is not possible. 

    Finally, we consider the case $a=1$. Solving \Cref{eq:csc_ricci_identities}, and consecutively the Gauss-Codazzi equations, we find that the hypersurface needs to have two different constant principal curvatures. In the classification of \Cite{takagi1973}, this hypersurface is of type A1. Then we know (as well as by a frame computation) that this hypersurface also cannot have constant sectional curvature. This concludes the proof. 
\end{proof}

\section*{Acknowledgements}

We would like to thank Dr M. Anarella for the fruitful discussions and poignant observations during seminars. Further, we would like to thank prof. L. Vrancken and prof J. Van der Veken for their support.

\pdfbookmark{Bibliography}{Bibliography}
\printbibliography

\end{document}

%% file: packages.tex
\usepackage[backend=biber, style = numeric, doi=false,isbn=false,url=false,eprint=false]{biblatex}
\renewbibmacro{in:}{}
\DeclareFieldFormat{pages}{#1}
\DeclareFieldFormat[article,periodical]{volume}{\mkbibbold{#1}}

\usepackage{hyperref}
\usepackage[noabbrev,capitalise]{cleveref} 

\newtheorem{mainthm}{Theorem}

\newtheorem{theorem}{Theorem}[section]
\newtheorem{lemma}[theorem]{Lemma}
\newtheorem{proposition}[theorem]{Proposition}

\theoremstyle{definition}
\newtheorem{definition}[theorem]{Definition}
\newtheorem{remark}{Remark}

\numberwithin{equation}{section}

\let\svqty\qty
\usepackage{physics}
\let\qty\svqty

\usepackage[shortlabels]{enumitem}
\usepackage[numbered]{bookmark}

\usepackage{dsfont} 

\usepackage{todonotes}

\newcommand*{\complexs}{\mathbb{C}}
\newcommand{\quaternions}{\mathbb{H}}

\newcommand*{\set}[1]{\l\{ #1 \r\}} 

\renewcommand{\l}{\left}
\renewcommand{\r}{\right}

\renewcommand*{\epsilon}{\varepsilon}

\let\phii\phi
\renewcommand*{\phi}{\varphi}

\renewcommand*{\tilde}{\widetilde}

\newcommand{\identity}{\mathds{1}} 
\newcommand{\CP}{{\complexs P}}
\newcommand{\sphere}[1]{\mathbb{S}^{#1}}

\newcommand{\lift}[1]{\widetilde{#1}}

\newcommand{\Jo}{J_1}

\newcommand{\Dtwo}{\mathcal{D}_1^2}
\newcommand{\Dfour}{\mathcal{D}_2^4}

\newcommand{\FS}[1]{{{#1}_{1}} }
\newcommand{\FSm}{g_{1}}
\newcommand{\FSnabla}{\nabla^{1}}

\DeclareMathOperator{\SU}{SU}
\DeclareMathOperator{\U}{U}
\DeclareMathOperator{\Sp}{Sp}

\DeclareMathOperator{\Iso}{Iso}

\DeclareMathOperator{\Span}{Span}